\documentclass[a4paper, 11pt, oneside]{amsart}

\usepackage[hmargin=3cm, vmargin=2.cm]{geometry}

\usepackage[dvipsnames]{xcolor}

\usepackage{amssymb}
\usepackage{amsmath}
\usepackage{graphicx}
\usepackage{mathtools}
\usepackage{amsthm}
\usepackage{amstext}
\usepackage[cmtip,all]{xy}
\usepackage{nicematrix}
\usepackage{enumitem}
\usepackage{mathrsfs}
\usepackage{tikz}
\usepackage{float}
\tikzset{
  vtx/.style={circle, draw, line width=0.6pt, fill=white, minimum size=3mm, inner sep=0pt},
  edg/.style={line width=0.6pt}
}
\usepackage[
  backend=biber,
  style=alphabetic,
  giveninits=true,
  maxbibnames=4,
  doi=false,
  isbn=false,
  url=false
]{biblatex}
\usepackage[colorlinks=true,linkcolor=blue,citecolor=ForestGreen]{hyperref}

\newcommand{\N}{\mathbb{N}}
\newcommand{\A}{{\mathbb{A}}}
\newcommand{\R}{{\mathbb{R}}}
\renewcommand{\S}{{\mathbb{S}}}
\newcommand{\Z}{{\mathbb{Z}}}
\renewcommand{\P}{{\mathbb{P}}}
\newcommand{\SL}{\operatorname{SL}}
\renewcommand{\O}{\operatorname{O}}

\newcommand{\Ccal}{{\mathcal{C}}}
\newcommand{\Gcal}{{\mathcal{G}}}
\newcommand{\Fcal}{\mathcal{F}}
\newcommand{\Vcal}{\mathcal{V}}
\newcommand{\Scal}{\mathcal{S}}

\newcommand{\PGL}{\operatorname{PGL}}

\newtheorem{theorem}{Theorem}[section]
\newtheorem{lemma}[theorem]{Lemma}
\newtheorem{proposition}[theorem]{Proposition}
\theoremstyle{definition}
\newtheorem{definition}[theorem]{Definition}

\newtheorem{corollary}[theorem]{Corollary}
\newtheorem{question}{Question}
\newtheorem{remark}[theorem]{Remark}

\numberwithin{equation}{section}

\begin{document}

\title{Finiteness of integral representations on 2-perfect truncation polytopes}

\author[S. Ko]{Sunghwan Ko}
\address{Department of Mathematical Sciences, Seoul National University, Seoul 08826, South Korea}
\email{sym2346@snu.ac.kr}

\date{\today}

\begin{abstract}
Let $P$ be a compact hyperbolic Coxeter truncation polytope of dimension $d\ge 3$, and let $\Gamma$ be the orbifold fundamental group of the associated Coxeter orbifold $\mathcal{O}_P$. Let $\mathscr{G}(\Gamma,G)$ be the geometric component containing the holonomy representation in $\operatorname{Hom}(\Gamma,G)/G$. $\mathscr{G}(\Gamma,G)$ is identified with the deformation space of properly convex real projective structures on the Coxeter orbifold $\mathcal{O}_P$. We prove that $\mathscr{G}(\Gamma,G)$ contains only finitely many integral representations. The same conclusion holds more generally for irreducible, large, $2$-perfect truncation polytopes.
\end{abstract}

\maketitle

\tableofcontents

\section{Introduction}

Let $M$ be a closed hyperbolic $d$--orbifold and let $\Gamma$ be its orbifold fundamental group. The holonomy representation $\rho_0 : \Gamma \longrightarrow \O^+(d,1)$ of $M$ may be viewed as a representation into $G := \SL^{\pm}(d+1,\R)$ via the standard embedding $\O^+(d,1)\hookrightarrow \SL^{\pm}(d+1,\R)$. Let $\mathscr{G}(\Gamma,G)$ denote the connected component containing $[\rho_0]$ of the locus in $\operatorname{Hom}(\Gamma,G)/G$ consisting of holonomy representations (up to conjugation) of properly convex real projective structures on $M$. We refer to $\mathscr{G}(\Gamma,G)$ as the \emph{geometric component}.

A natural arithmetic question is the following.

\begin{question}\label{q:1}
How many integral representations are there in $\mathscr{G}(\Gamma,G)$?
\end{question}

For $G=\SL^{\pm}(d+1,\R)$, we write $G(\Z)$ for $\SL^{\pm}(d+1,\Z)$. We say that a representation $\rho:\Gamma\to G$ is \emph{integral} if there exists $g\in G$ such that
\begin{equation*}
g\rho(\Gamma)g^{-1}\subset G(\Z).
\end{equation*}

Thus, the question concerns the integral points in a geometric deformation component containing the holonomy representation.

The problem already exhibits subtle behavior in dimension two. When $d=2$ and $G=\SL^{\pm}(3,\R)$, the geometric component $\mathscr{G}(\Gamma,G)$ coincides with the Hitchin component, or equivalently, with the deformation space of convex real projective structures on the underlying $2$--orbifold \cite{Hitchin1992LieGroupsTeichmuller,Goldman1990ConvexRP2,ChoiGoldman1993Closedness,ChoiGoldman05}. In this case, Question~\ref{q:1} asks whether such a natural geometric deformation space can contain infinitely many integral points.

The answer depends strongly on the orbifold. Long--Reid--Thistlethwaite \cite{LongReidThistlethwaite2011SL3Z} showed that, for the turnover triangle orbifold group
\begin{equation*}
\Delta=\langle a,b \mid a^3=b^3=(ab)^4=1\rangle,
\end{equation*}
the geometric component $\mathscr{G}(\Delta,G)$ contains infinitely many integral representations. By contrast, for every compact Coxeter triangle group the corresponding geometric component contains only finitely many integral representations (see \cite{ChoiChoi2015Definability}). This leads naturally to the following question in higher dimensions.

\begin{question}\label{q:2}
Let $P$ be a compact hyperbolic Coxeter $d$--polytope with $d\ge 2$, and let $\Gamma=\pi_1^{\mathrm{orb}}(P)$
be the associated Coxeter orbifold group. Are there only \emph{finitely many} integral representations in the geometric component $\mathscr{G}(\Gamma,G)$ containing the hyperbolic holonomy class?
\end{question}

The geometric component $\mathscr{G}(\Gamma,G)$ is already quite explicit in dimension two. If $P$ is a compact hyperbolic Coxeter $k$--gon with $b$ right angles and $\Gamma$ is the orbifold fundamental group of $P$, then Goldman proved that $\mathscr{G}(\Gamma,G)$ is an open cell of dimension $3k-8-b$ \cite{Goldman1977Affine}. Note that one can have a polygon from $2$--simplex by successively truncating vertices. Since polygons may also be obtained from a simplex by successive vertex truncations, it is natural to consider the higher-dimensional analogue of this construction.

In this paper we therefore focus on \emph{truncation polytopes}, namely Coxeter polytopes obtained from a simplex by successive vertex truncations. This class is a natural one in which to study Question~\ref{q:2} for two reasons. First, their deformation spaces are well understood, both in dimension and in parametrization: for Coxeter truncation polytopes of dimension at least $3$, the geometric component is homeomorphic to an open ball of dimension $e_+-d$, where $e_+$ denotes the number of non-right-angled ridges. This was proved by Choi--Lee--Marquis in dimensions $d \ge 4$ and by Marquis in dimension $3$ \cite{ChoiLeeMarquis22,Marquis10_EspacePolyedres}. Second, truncation polytopes admit a recursive decomposition into simpler Coxeter pieces.

\begin{theorem}\label{thm:1.1}
Let $P$ be a compact hyperbolic Coxeter truncation polytope of dimension $d\ge 3$, and let $M$ be the associated Coxeter $d$--orbifold with $\Gamma:=\pi_1^{\mathrm{orb}}(M)$.
Then the geometric component $\mathscr{G}(\Gamma,\SL^{\pm}(d+1,\R))$ contains only finitely many integral representations.
\end{theorem}

A notable feature of Theorem~\ref{thm:1.1} is that the geometric component may have arbitrarily large dimension, whereas its integral locus remains finite. In particular, the theorem reveals a strong arithmetic finiteness phenomenon inside a geometrically flexible deformation space.

More generally, the proof yields a finiteness theorem for irreducible, large, $2$--perfect truncation polytopes, including noncompact examples. The compact hyperbolic case follows by identifying the deformation space of the Coxeter polytope with the corresponding geometric component of holonomy representations.

The proof uses the recursive structure of truncation polytopes. Cutting along suitable codimension-one pieces reduces the problem to simpler pieces, and the corresponding deformation parameter is described by bending along the cut. The key point is that integrality forces this bending parameter to remain bounded. Combined with the corresponding boundedness for the simpler pieces, this yields finiteness of the integral locus.

Question~\ref{q:2} remains open beyond the class treated in this paper, and the difficulty takes different forms depending on the dimension and on the combinatorial type of $P$. 

In dimension two, every compact hyperbolic Coxeter polygon is a truncation polytope, but the proof of the theorem above does not extend to that setting. Indeed, the parameters for the simpler pieces are not bounded. Hence the splitting--and--gluing procedure used in higher dimensions does not provide a uniform bound for the parameters of the resulting polygon. Thus the two--dimensional case remains open even within the truncation class.

A different issue arises in dimensions $d\ge 3$. For truncation polytopes, the geometric component admits an explicit parametrization, which plays a crucial role in analyzing integrality. However, for compact hyperbolic Coxeter polytopes outside the truncation class, no such parametrization is known in general. This leaves Question~\ref{q:2} open for compact hyperbolic Coxeter polytopes outside the truncation class.

\subsection{Organization of the paper}
Section~2 reviews Vinberg's theory and fixes notation, including cyclic-product coordinates and the integrality criterion. Section~3 describes the deformation spaces of simplices and once truncated simplices. Section~4 studies splitting and gluing along prismatic circuits and parametrizes deformation spaces of truncation polytopes. Section~5 we prove the main theorem by showing that the bending parameter is bounded on each fiber of the integral locus and then applying induction.

\subsection*{Acknowledgements}
I am grateful to my advisor, Gye--Seon Lee, for introducing me to this problem and for many helpful conversations. I would also like to thank Seunghoon Hwang and Donghyun Lee for their helpful comments and suggestions, which greatly improved the readability of this paper.

This work was supported by the National Research Foundation of Korea(NRF) grant funded by the Korea government(MSIT)(No. RS-2023-00252171).

\section{Preliminaries}\label{sec:prelim}

In this section, we recall Vinberg's theory on reflection groups, which will be used throughout the paper.

\subsection{Coxeter polytopes and Vinberg theory}

\begin{definition}\label{def:2.1}
Let $S$ be a finite set. A \emph{Coxeter matrix} on $S$ is a symmetric $S \times S$ matrix 
$M = (M_{st})_{s,t \in S}$ with entries $M_{st} \in \{1, 2, \dots, \infty\}$ such that:
\begin{enumerate}
    \item $M_{ss} = 1$ for all $s \in S$;
    \item $M_{st} \neq 1$ for all distinct $s, t \in S$.
\end{enumerate}
A \emph{Coxeter group} $W_S$ associated with $M$ is defined by the presentation
\begin{equation*}
   W_S = \langle \, S \mid (st)^{M_{st}} = 1 \text{ for all } s,t \in S \text{ with } M_{st}\neq\infty \, \rangle . 
\end{equation*}
The \emph{rank} of $W_S$ is the cardinality $|S|$ of $S$.
\end{definition}

\begin{definition}\label{def:2.2}
Let $W_S$ be a Coxeter group with generating set $S$ and Coxeter matrix $M = (M_{st})_{s,t \in S}$. The \emph{Coxeter diagram} of $W_S$ is the labeled graph defined as follows:
\begin{enumerate}
    \item The set of nodes is $S$.
    \item Two nodes $s,t \in S$ are connected by an edge $\overline{st}$ if and only if $M_{st}>2$.
    \item The edge $\overline{st}$ is labeled by the number $M_{st}$ if $M_{st}>3$.
\end{enumerate}
\end{definition}

Suppose $W_S$ is a Coxeter group, and let $\mathscr{D}_{W_S}$ denote the Coxeter diagram of $W_S$. A Coxeter group $W_S$ is said to be \emph{irreducible} if its diagram $\mathscr{D}_{W_S}$ is connected. By Margulis--Vinberg \cite[Thm~1]{MargulisVinberg00}, every irreducible Coxeter group $W_S$ belongs to one of the following classes:
\begin{enumerate}
    \item \emph{Spherical} if $W_S$ is finite.
    \item \emph{Affine} if $W_S$ is virtually abelian and infinite.
    \item \emph{Large} if there exists a surjection $W_S \twoheadrightarrow F_2$ onto a free group $F_2$ of rank two. 
\end{enumerate}

For each subset $S' \subseteq S$, we can consider the Coxeter group $W_{S'}$ associated with the Coxeter matrix $M'=(M_{st})_{s,t\in S'}$. The group $W_{S'}$ admits a natural injective homomorphism $W_{S'} \rightarrow W_S$, which identifies $W_{S'}$ with the subgroup of $W_S$ generated by $S'$. Such subgroups of $W_{S}$ are called the \emph{standard subgroups} of $W_S$.

Let $\S : \R^{d+1} \setminus \{0\} \to (\R^{d+1} \setminus \{0\}) / \R_{>0}$ be the canonical projection. The image $\S(\R^{d+1} \setminus \{0\})$ is the projective sphere, which we denote by $\S^d$. By a slight abuse of notation, for any subset $W \subset \R^{d+1}$, we write $\S(W)$ to denote the image $\S(W \setminus \{0\})$. The group of linear automorphisms of $\S^d$ is precisely $\SL^{\pm}(d+1, \R)$, the group of linear automorphisms of $\R^{d+1}$ with determinant $\pm 1$.

\begin{definition}\label{def:2.3}
A \emph{projective reflection} $\sigma$ on $\S^d$ is an involution that fixes a hyperplane in $\S^d$ pointwise.
\end{definition}

More concretely, a projective reflection $\sigma$ is induced by a linear map of the form
\begin{equation*}
    \sigma = \mathrm{Id} - \alpha \otimes b,
\end{equation*}
where $\alpha \in (\R^{d+1})^*$, $b \in \R^{d+1}$, and $\alpha(b) = 2$. The projective hyperplane $\S(\ker \alpha)$ is called the \emph{support} of $\sigma$, and the point $[b] \in \S^d$ is called its \emph{pole}.

Let $\pi : \S^d \to \R\P^d$ be the natural projection. An \emph{affine chart} $\A \subset \S^d$ is the preimage of an affine open subset of $\R\P^d$, namely the complement of the projective hyperplane associated to some linear hyperplane $H \subset \R^{d+1}$. A subset $P \subset \S^d$ is \emph{convex} if $P = \S(W)$ for some convex cone $W \subset \R^{d+1}$. A convex subset $P \subset \S^d$ is \emph{properly convex} if its closure $\overline{P}$ is contained in some affine chart $\A \subset \S^d$.

A properly convex subset $P \subset \S^d$ is called a \emph{projective polytope} if it has nonempty interior $\mathring{P}$ and can be written as
\begin{equation*}
    P = \bigcap_{i=1}^N \S \{ v \in \R^{d+1} \mid \alpha_i(v) \leq 0 \},
\end{equation*}
where each $\alpha_i \in (\R^{d+1})^*$. We always assume that such a representation is minimal, in the sense that no half-space $\S \{ v \in \R^{d+1} \mid \alpha_i(v) \leq 0 \}$ contains the intersection of the others. A codimension-one face of $P$ is called a \emph{facet}, and a codimension-two face of $P$ is called a \emph{ridge}. Two facets $s, t$ of $P$ are said to be \emph{adjacent} if their intersection $s \cap t$ is a ridge.

\begin{definition}\label{def:2.4}
    A \emph{Coxeter polytope} is a pair $(P,(\sigma_s=\mathrm{Id}-\alpha_s\otimes b_s)_{s\in S})$ of a projective polytope and a collection of projective reflections such that:
\begin{enumerate}
    \item For each $s \in S$ the support of $\sigma_s$ is the kernel of $\alpha_s$.
    \item For every pair of distinct facets $s,t$ of $P$, 
    \begin{enumerate}
        \item $\alpha_s(b_t)
        $ and $\alpha_t(b_s)$ are both negative, or both zero. 
        \item $\alpha_s(b_t)\alpha_t(b_s)\geq 4$ or $\alpha_s(b_t)\alpha_t(b_s)=4\cos^2\frac{\pi}{m_{st}}$ for some $m_{st} \in \N \setminus \{0,1\}$.
    \end{enumerate}
\end{enumerate}
\end{definition}

We often denote a Coxeter polytope simply by $P$. Let $s,t$ be distinct facets of $P$, and $U$ be the intersection of $\ker(\sigma_s)$ and $\ker(\sigma_t)$. The composition $\sigma_s\sigma_t$ restricts to the identity on $U$. It induces an element ${\sigma_s\sigma_t}\vert_{\R^{d+1}/U}$ of $\SL(2,\R)$, which is conjugate to one of the following matrices by \cite[Proposition~6]{Vinberg71a}:
\begin{enumerate}
    \item
    $\begin{pmatrix}
    \lambda & 0\\
    0 & \lambda^{-1}
    \end{pmatrix}$
    \quad \text{for some } $\lambda>0$,
    if $\alpha_s(b_t)\alpha_t(b_s)>4$.

    \vspace{0.2cm}
    \item
    $\begin{pmatrix}
    1 & 1\\
    0 & 1
    \end{pmatrix}$, if $\alpha_s(b_t)\alpha_t(b_s)=4$.

    \vspace{0.2cm}
    
    \item
    $\begin{pmatrix*}[r]
    \cos\bigl(\tfrac{2\pi}{m_{st}}\bigr) & -\sin\bigl(\tfrac{2\pi}{m_{st}}\bigr)\\[0.2em]
    \sin\bigl(\tfrac{2\pi}{m_{st}}\bigr) & \cos\bigl(\tfrac{2\pi}{m_{st}}\bigr)
    \end{pmatrix*}$, if $\alpha_s(b_t)\alpha_t(b_s)=4\cos^2\!\bigl(\tfrac{\pi}{m_{st}}\bigr)$.
\end{enumerate}

For each pair of adjacent facets $s,t$ of $P$, the \emph{dihedral angle} of the ridge $s\cap t$ is defined to be $\frac{\pi}{m_{st}}$ in case~(3), and 0 otherwise. We define the Coxeter matrix $M$ associated to $P$ by setting $M=(M_{st})_{s,t\in S}$ in case~(3), and $M_{st}=\infty$ in the other cases. We denote the Coxeter group associated to $M$ by  $W_P$.

If $P$ is a Coxeter polytope and $f$ is a proper face of $P$, then we denote $S_f\coloneqq\{s\in S \ \vert f \subset s\}$ and $W_f\coloneqq W_{S_f}$.

\begin{theorem}[\cite{Bourbaki68} and \cite{Vinberg71a}]\label{thm:2.5}
    Let P be a Coxeter polytope in $\S^d$  with the associated Coxeter group $W_P$ and let $\Gamma_P$ be the subgroup of $\operatorname{SL}^{\pm}(d+1,\R)$ generated by the projective reflections $(\sigma_s)_{s\in S}$. Then the following hold:
    \begin{enumerate}
        \item The assignment $\rho(s)=\sigma_s$ defines an isomorphism $\rho : W_P \rightarrow \Gamma_P \subset \SL^{\pm}(d+1,\R)$.
        \item $\gamma(\mathring{P})\cap \mathring{P} = \varnothing$ for $\gamma\in \Gamma_P\setminus e$.
        \item $\Ccal_P\coloneqq\bigcup\{\gamma\cdot P ~\vert~\gamma\in\Gamma_P\}$ is a convex subset of $\S^d$.
        \item The group $\Gamma_P$ is a discrete subgroup of $\SL^{\pm}(d+1,\R)$ acting properly discontinuously on $\Omega_P\coloneqq \mathring{\Ccal_P}$.
        \item An open proper face $f$ of $P$ lies in $\Omega_P$ if and only if the Coxeter group $W_f$ is finite.
    \end{enumerate}
\end{theorem}

The group $\Gamma_P$ is called the \emph{projective Coxeter group} of $P$. Theorem~\ref{thm:2.5} tells us that the quotient $\Omega_P /\Gamma_P$ is a convex real projective Coxeter orbifold and $\rho$ is the holonomy representation of $\Omega_P/\Gamma_P$. The holonomy representation $\rho$ is called the \emph{Vinberg representation}.

Let $P$ be a Coxeter polytope of dimension $d$. We say that $P$ is \emph{perfect} if the stabilizer $W_v$ is finite for every vertex $v$ of $P$ and \emph{$2$-perfect} if the stabilizer $W_f$ is finite for each edge $f$ of $P$.

We also express $2$-perfectness of a Coxeter polytope $P$ in terms of the \emph{links} of the vertices of $P$. For any vertex $v$ of $P$, observe that the vector space $\langle v \rangle$ generated by $v$ is stabilized by every reflection $\sigma_s$ with $s \in S_v$. It follows that each $\sigma_s$ induces a reflection on $\S(\R^{d+1}/\langle v\rangle)$. We define the link $P_v$ of $P$ at $v$ by
\begin{equation*}
    P_v = \bigcap_{s \in S_v} \, \mathbb{S}\bigl\{ x \in \R^{d+1}/\langle v \rangle \;\big|\; \alpha_s(x) \leq 0 \bigr\}.
\end{equation*}

By the definition of the link $P_v$, every vertex in $P_v$ corresponds to an edge of $P$ containing $v$. Therefore $P$ is $2$--perfect if and only if each link $P_v$ is perfect.

\begin{definition}\label{def:2.6}
    Let $S$ be a finite set. A \emph{Cartan matrix} $A_S=(A_{ij})_{i,j \in S}$ on $S$ is a $\vert S\vert\times \vert S\vert$ matrix satisfying the following conditions:
    \begin{enumerate}
        \item $A_{ii}=2$ for all $i\in S$, and $A_{ij}\leq 0$ for $i\neq j$. 
        \item $A_{ij}=0$ if and only if $A_{ji}=0$.
        \item For $i\neq j$, either $A_{ij}A_{ji}\geq 4$, or $A_{ij}A_{ji}=4\cos^2(\pi/m_{ij})$ for some $m_{ij}\in \{3,4,\ldots \}$.
    \end{enumerate}
\end{definition}

A Cartan matrix $A$ is said to be \emph{reducible} if, after reordering its indices, it can be written in block diagonal form; otherwise, $A$ is called \emph{irreducible}. For an irreducible Cartan matrix $A$, we can write $A=2\mathrm{Id}-M$ for some matrix $M$ with nonnegative entries. Therefore the Perron--Frobenius theorem ensures the existence of the smallest modulus eigenvalue $\lambda_A$. 

An irreducible Cartan matrix $A$ is classified into one of three types: \emph{positive}, \emph{zero}, or \emph{negative type} according to the sign of the eigenvalue $\lambda_A$. Every Cartan matrix $A$ decomposes into the direct sum of its irreducible components, each of which is of positive, zero, or negative type. Rearranging indices, we denote by $A^+$, $A^0$, and $A^-$ the connected components of positive, zero, and negative type, respectively.

For a Coxeter polytope $(P,(\sigma_s=\mathrm{Id}-\alpha_s\otimes b_s)_{s\in S})$ in $\S^d$,
we define a $(\R_+)^S$-action on the tuples $(\alpha_s,b_s)_{s\in S}$ in $(V^*\times V)^{S}$
\begin{equation*}
(\lambda_s)_{s\in S}\cdot (\alpha_s,b_s)_{s\in S} \mapsto (\lambda_s\alpha_s,\lambda_s^{-1}b_s)_{s\in S}.
\end{equation*}
The corresponding Vinberg representation is determined only up to this action. This action induces the usual diagonal action on Cartan matrices. Accordingly, we define an equivalence relation on Cartan matrices by
\begin{equation*}
A\sim A' \quad \text{if and only if} \quad A'=DAD^{-1}
\end{equation*}
for some positive diagonal matrix $D$.

Given a Coxeter polytope $(P,(\sigma_s=\mathrm{Id}-\alpha_s\otimes b_s)_{s\in S})$, the matrix $A_P\coloneqq(\alpha_s(b_t))_{s,t\in S}$ satisfies the conditions of Definition~\ref{def:2.6}. We call $A_P$ the Cartan matrix of $P$. Then the equivalence class $[A_P]$ is independent of the choice of defining reflections of $P$. Any representative of this class will be called a Cartan matrix of $P$. A Cartan matrix $A_P$ is irreducible if and only if the Coxeter group $W_P$ is irreducible. In this case, we say that $P$ is \emph{irreducible}.

\begin{definition}\label{def:2.7}
Let $P$ be a Coxeter polytope in $\mathbb{S}^d$ and $A_P$ be the associated Cartan matrix.  
We say that $P$ is \emph{elliptic} if $A_P = A_P^+$, 
\emph{parabolic} if $A_P = A_P^0$ and $\operatorname{rank}(A_P) = d$,  
and \emph{loxodromic} if $A_P = A_P^-$ and $\operatorname{rank}(A_P) = d+1$.
\end{definition}

The following theorem shows that a Cartan matrix $A$ satisfying a certain condition determines a unique irreducible Coxeter polytope.

\begin{theorem}[\cite{Vinberg71a}]\label{thm:2.8}
    Let $A$ be a Cartan matrix of size $N \times N$. Suppose that $A$ is irreducible, of negative type, and of rank $d+1$. Then there exists a unique Coxeter $d$-polytope $P$, up to automorphism of $\mathbb{S}^d$, such that $A = A_P$.
\end{theorem}

\subsection{Deformation spaces of labeled polytopes}

Let $P$ be a projective polytope in $\S^d$. The face poset $\Fcal(P)$ of $P$ is the partially ordered set of all faces of $P$, ordered by inclusion. Two projective $d$-polytopes $P$ and $P'$ are said to be \emph{combinatorially equivalent} if there exists a bijection $\phi : \Fcal(P) \to \Fcal(P')$ preserving inclusion relations. A \emph{combinatorial polytope} is a combinatorial equivalence class of a projective polytope.

Let $\Gcal$ be a combinatorial polytope. Assigning an element $m \in \{2,3,\ldots,\infty\}$ to each ridge of $\Gcal$ determines a function from the set of ridges to $\{\pi/m \ \vert \ m=2,3,\ldots,\infty\}$. We call this assignment a \emph{ridge labeling}. A \emph{labeled polytope} is a combinatorial polytope together with a ridge labeling.

Let $\Gcal$ be a labeled polytope in $\S^d$ . A \emph{Coxeter polytope realizing} $\Gcal$ is a pair $(P,\phi)$, where $P$ is a Coxeter polytope  and  $\phi$ is a poset isomorphism from the face poset of $P$ to that of $\Gcal$, such that for each ridge $r$ of $\Gcal$, the label of $r$ agrees with the dihedral angle of the corresponding ridge $\phi^{-1}(r)$ of $P$. Two Coxeter polytopes $(P,\phi)$ and $(P',\phi')$ realizing $\Gcal$ are isomorphic if there exists an automorphism $\psi$ of $\S^d$ such that $\psi(P)=P'$ and $\hat{\psi}\circ \phi=\phi'$, where $\hat{\psi}$ is the poset isomorphism between $\Fcal(P)$ and $\Fcal(P')$ induced by $\psi$. 
 
We often omit $\phi$ and simply write $P$ for the pair $(P,\phi)$. Its isomorphism class in $\Ccal(\Gcal)$ is denoted by $[P]$. 

\begin{definition}\label{def:2.9}
    Let $\Gcal$ be a labeled polytope. The \emph{deformation space} $\Ccal(\Gcal)$ of $\Gcal$ is the space of isomorphism classes of Coxeter polytopes realizing $\Gcal$.
\end{definition}

Suppose that $\Gcal$ is a labeled $d$-polytope, and let $P$ be a Coxeter polytope realizing~$\Gcal$. Let $\psi \in \PGL(d+1,\R)$ be a projective automorphism of $\S^d$, and choose a representative
$g\in\SL^\pm(d+1,\R)$ of $\psi$. Then $\psi$ acts on $P$ through the natural action of $\PGL(d+1,\R)$ on $\S^d$, so that
\begin{equation*}
    \psi(P)=g(P).
\end{equation*}

If $(\sigma_s)_{s\in S}$ denotes the reflections associated with the facets of $P$, then the induced action on the corresponding realization is given by conjugation:
\begin{equation*}
    g\cdot (P,(\sigma_s)_{s\in S}):= \bigl(g(P),(g\sigma_s g^{-1})_{s\in S}\bigr).
\end{equation*}

Therefore, the deformation space $\Ccal(\Gcal)$ is defined as the space of Coxeter polytopes realizing $\Gcal$, modulo the conjugation action of $\SL^\pm(d+1,\R)$.

\begin{remark}\label{rem:2.10}
Let $\Gcal$ be a labeled polytope, and let $M$ be the associated Coxeter orbifold. Then the deformation space $\Ccal(\Gcal)$ may be regarded as the deformation space of convex real projective structures on $M$.
\end{remark}

\begin{remark}[Compact Coxeter case]\label{rem:2.11}
Assume that $P$ is a compact hyperbolic Coxeter polytope realizing $\Gcal$, and let $M$ be the associated Coxeter orbifold. Set $\Gamma=\pi_1^{\mathrm{orb}}(M)\cong W_\Gcal$. By the openness theorem of Koszul \cite{Koszul68_DeformationsConnexionsLocalementPlates} and the closedness theorem of Benoist \cite{Benoist05_ConvexesDivisiblesIII}, the subset of $\operatorname{Hom}(\Gamma,G)/G$ consisting of holonomy classes of properly convex real projective structures on $M$ is a union of connected components. Hence the geometric component $\mathscr{G}(\Gamma,G)$ containing the hyperbolic holonomy class $[\rho_0]$ can be identified, via the holonomy map, with the connected component $\Ccal_0(\Gcal)\subset \Ccal(\Gcal)$ containing the hyperbolic point. We will therefore identify $\mathscr{G}(\Gamma,G)$ with $\Ccal_0(\Gcal)$ throughout the paper.
\end{remark}

Let $\Gcal$ be a labeled polytope and $P$ be a Coxeter polytope realizing $\Gcal$. We next define several notions for $\Gcal$ by saying that $\Gcal$ satisfies a given property if a Coxeter polytope realizing $\Gcal$ satisfies it. For example, a labeled polytope $\Gcal$ is irreducible if a Coxeter polytope realizing $\Gcal$ is irreducible.
 
We recall cyclic products, which will be used to parametrize Coxeter polytopes realizing a given labeled polytope $\Gcal$.

Let $A = (a_{ij})$ be an $N \times N$ Cartan matrix, and let $S = \{1, 2, \dots, N\}.$
A $k$-tuple of elements in $S$ is called a \emph{$k$-circuit} of $S$. For any $k$-circuit $\Ccal = (i_1, i_2, \dots, i_k)$ in $S$, we define the \emph{cyclic product} of $A$ associated with $\Ccal$ as

\begin{equation*}
    \Ccal(A) = a_{i_1 i_2} a_{i_2 i_3} \ldots a_{i_k i_1}.
\end{equation*}

\begin{proposition}[\cite{Vinberg71a}, Prop.~16]\label{prop:2.12}
     Let $A$ and $A'$ be two Cartan matrices. Then $A$ and $A'$ are equivalent if and only if all their cyclic products are equal.
\end{proposition}

\begin{theorem}\label{thm:2.13}
    Let $\Gcal$ be a labeled $d$-polytope, and let $\Ccal(\Gcal)$ be its deformation space. Then the following are equivalent:
    \begin{itemize}
        \item $[P]=[P']$ in $\Ccal(\Gcal)$;
        \item the Cartan matrices $A_P$ and $A_{P'}$ are equivalent;
        \item all cyclic products of $A_P$ and $A_{P'}$ are equal.
    \end{itemize}
\end{theorem}

The theorem shows that the deformation space $\Ccal(\Gcal)$ of a labeled polytope $\Gcal$ can be described in terms of cyclic products of Cartan matrices.

A $k$--circuit $C$ in the generating set is said to be \emph{relevant} if one of the following holds:
\begin{itemize}
    \item[(i)] $k\ge 3$ and $C$ corresponds to a cycle in the underlying graph of the Coxeter diagram $\mathscr{D}_W$;
    \item[(ii)] $k=2$ and $C$ corresponds to an edge labeled by $\infty$ in $\mathscr{D}_W$.
\end{itemize}

Let $A$ be a Cartan matrix associated with a Coxeter polytope $P$ realizing the labeled polytope $\Gcal$. Then the cyclic product $C(A)$ associated with any non--relevant circuit is automatically determined:

\begin{itemize}
    \item if $k=1$, then $C(A)=2$;
    \item if $k=2$ and the corresponding edge has finite label $m$ (including $m=2$ when there is no edge), then $C(A)=a_{ij}a_{ji}=4\cos^2(\pi/m)$;
    \item if $k\ge 3$, then some consecutive pair is nonadjacent in the underlying graph, hence $a_{ij}=a_{ji}=0$ and therefore $C(A)=0$.
\end{itemize}

Relevant circuits give rise to nontrivial cyclic products, and these cyclic products provide effective parameters for the point $[P]\in\Ccal(\Gcal)$.

For a circuit $C=(i_1,\ldots,i_k)$ and its opposite circuit $\overline{C}=(i_k,\dots,i_1)$, we define the \emph{normalized cyclic product} by the following:

\begin{equation*}
R_{C}(A_P)\ :=\ \log\!\Big(\frac{C(A_P)}{\overline{C}(A_P)}\Big).
\end{equation*}

When a circuit $C=(i_1,\ldots,i_k)$ has no $\infty$--edges, the corresponding normalized cyclic product is determined by the cyclic product $C(A)$.

\subsection{Irreducibility and integrality of Vinberg representations}
Let $\Gcal$ be a labeled $d$-polytope and $\rho:W_P\to \SL^\pm(d+1,\R)$ associated with a realization $P$ of $\Gcal$. We say that $\rho$ is \emph{irreducible} if $\rho(W_P)$ preserves no nontrivial proper linear subspace of $\R^{d+1}$.

The following criterion relates the irreducibility of  the representation $\rho$ and  the Cartan matrix $A_P$.

\begin{theorem}[\cite{Vinberg71a}]\label{thm:2.14}
    Let $P$ be a Coxeter polytope in $\S^d$  and $W_P$ be its Coxeter group. Then the following are equivalent:
    \begin{itemize}
        \item The representation $\rho : W_P\rightarrow \operatorname{SL}^{\pm}(d+1,\R)$ is irreducible.
        \item $W_P$ is irreducible and the family of polars $(b_s)_{s\in S}$ spans $\R^{d+1}$.
        \item $W_P$ is irreducible and the Cartan matrix $A_P$ of $P$ is of rank $d+1$.
    \end{itemize}
\end{theorem}

\begin{corollary}\label{cor:2.15}
Let $P$ be a Coxeter polytope. Assume that $W_P$ is infinite. Then the Vinberg representation $\rho$ is irreducible if and only if $P$ is irreducible and loxodromic.

In particular, for an irreducible, large labeled polytope $\Gcal$, every realization $P$ of $\Gcal$ yields an irreducible Vinberg representation $\rho$.
\end{corollary}

\begin{proof}
If $\rho$ is irreducible, then by Theorem~\ref{thm:2.14} the Cartan matrix $A_P$ is irreducible and has rank $d+1$. Since $W_P$ is infinite, $A_P$ is not of positive type. It is also not of zero type, so it must be of negative type. Hence $P$ is irreducible and loxodromic.

Conversely, if $P$ is irreducible and loxodromic, then $A_P$ is irreducible of rank $d+1$. Therefore Theorem~\ref{thm:2.14} implies that $\rho$ is irreducible.
\end{proof}

A real representation $\rho : W_P \to \operatorname{SL}^{\pm}(d+1,\R)$ is said to be \emph{absolutely irreducible} if its complexification $\rho_{\mathbb{C}}$ is irreducible. The following theorem ensures absolute irreducibility of the Vinberg representation.

\begin{theorem}[\cite{Vinberg71b}]\label{thm:2.16}
Let $P$ be an irreducible Coxeter polytope in $\mathbb{S}^d$, and let $\rho : W_P \to \operatorname{SL}^{\pm}(d+1,\R)$ be the Vinberg representation. If $W_P$ is large and $\rho$ is irreducible, then $\rho$ is absolutely irreducible.
\end{theorem}

The following lemma gives  a concrete and effective criterion for determining whether a Vinberg representation is integral.

\begin{lemma}[\cite{Vinberg71b}]\label{lem:2.17} Let $P$ be an irreducible, large Coxeter polytope. Then $\rho : W_P \rightarrow \operatorname{SL}^{\pm}(d+1,\R)$ is an integral representation if and only if every cyclic product lies in $\Z$. 
\end{lemma}

\begin{proof}
Since $P$ is an irreducible loxodromic Coxeter polytope, the corresponding representation $\rho : W_P \to \SL^\pm(d+1,\R)$ is absolutely irreducible by Corollary~\ref{cor:2.15} and Theorem~\ref{thm:2.16}. Therefore, by \cite[Lemma~6.2]{AudibertDoubaLeeMarquis25_ZariskiClosures}, the two conditions are equivalent.
\end{proof}

\section{Deformation spaces of simplices and once truncated simplices}\label{sec:simplex} This section relies on the classification results of \cite{Marquis10_EspacePolyedres} in dimension 3 and of \cite{ChoiLeeMarquis22} in dimensions $d \ge 4$. In the cases of simplices and once-truncated simplices that arise in the truncation setting, the deformation spaces are described explicitly in terms of cyclic products of Cartan matrices.

\subsection{Deformation spaces of 2-perfect simplices} 

We first introduce a deformation space of labeled triangles. Let $\Gcal$ be a labeled triangle, and $W_\Gcal$ is its Coxeter group. To parametrize the deformation space $\Ccal(\Gcal)$ of $\Gcal$, we use its cyclic products.

\begin{theorem}[{\cite[Proposition~4.3]{Marquis10_EspacePolyedres}}]\label{thm:3.1}\label{thm:marquis-triangle}
Let $\Gcal$ be a perfect labeled triangle whose edges are indexed by $1,2,3$.
Let $P$ be a realization of $\Gcal$ and $A_P$ its Cartan matrix.
\begin{enumerate}
    \item If $\Gcal$ has no right angle, then $\Ccal(\Gcal)$ is diffeomorphic to $\R$.
    Moreover, the map
    \begin{equation*}
        [P]\ \longmapsto\ R_{(1,2,3)}(A_P)
    \end{equation*}
    is a diffeomorphism from $\Ccal(\Gcal)$ onto $\R$.
    \item If $\Gcal$ has a right angle, then $\Ccal(\Gcal)$ is a singleton.
\end{enumerate}
\end{theorem}

\begin{proof}
Assume that $\Gcal$ has no right angle. Up to $\SL^\pm(3,\R)$-conjugation and diagonal rescaling, we may assume that the Cartan matrix $A_P$ of $P$ is of the form
\begin{equation*}
A_P=
\begin{pmatrix}
2 & -2\cos\frac{\pi}{m_{12}} & -2\cos\frac{\pi}{m_{13}}\\
-2\cos\frac{\pi}{m_{12}} & 2 & -2\cos\frac{\pi}{m_{23}}\,x\\
-2\cos\frac{\pi}{m_{13}} & -2\cos\frac{\pi}{m_{23}}\,x^{-1} & 2
\end{pmatrix}
,
\end{equation*}
for a unique parameter $x\in(0,\infty)$.
Substituting these entries, we compute
\begin{equation*}
R_{(1,2,3)}(A_P)=\log\!\left(\frac{a_{12}(x)a_{23}(x)a_{31}(x)}{a_{13}(x)a_{32}(x)a_{21}(x)}\right)=-2\log x.
\end{equation*}
Hence $[P]\mapsto R_{(1,2,3)}(A_P)$ gives a diffeomorphism $\Ccal(\Gcal)\cong\R$.

If $\Gcal$ has a right angle, then its Coxeter diagram contains no $3$-cycle (equivalently, no relevant circuit),
hence $\Ccal(\Gcal)$ is rigid.
\end{proof}

Let $\Gcal$ be an irreducible, large, $2$-perfect labeled simplex whose vertices are indexed by $1,2,3,4$. By the definition of $2$-perfectness, every vertex link $\Gcal_i$ of $\Gcal$ is perfect. Each vertex link has its own cyclic products, and hence its own normalized cyclic products.

We now parametrize $\Ccal(\Gcal)$ using the normalized cyclic products introduced above.
\begin{theorem}[{\cite[Proposition~4.29]{Marquis10_EspacePolyedres}}]\label{thm:3.2}
    Let $\Gcal$ be an irreducible, $2-$perfect, labeled $3$-simplex whose vertices are indexed by $1,2,3,4$.
    Then the deformation space $\Ccal(\Gcal)$ of $\Gcal$ is diffeomorphic to an open cell of dimension $b(\Gcal) = e_+-3\in\{0,1,2,3\}$.
\end{theorem}
\begin{proof}
Since $\Gcal$ is $2$-perfect, every vertex link $\Gcal_i$ is perfect, so $\Gcal$ cannot have an edge labeled by $\infty$. Thus $\Gcal$ falls into one of the five cases according to the number and position of right-angled edges. Label the facets by $\{1,2,3,4\}$ so that removing $i$ gives the Coxeter diagram of the link $\Gcal_i$. For $[P]\in\Ccal(\Gcal)$, write $A_P=(a_{ij})$.

\smallskip
\noindent
\emph{Case (1): no right-angled edge.}
In this case, each link $\Gcal_i$ is a non-right-angled perfect triangle, and hence each contributes one parameter. Set
\begin{equation*}
C_1=(2,3,4),\quad C_2=(4,3,1),\quad C_3=(1,2,4),\quad C_4=(1,2,3).
\end{equation*}
By the same argument as in Theorem~\ref{thm:3.1}, the four normalized cyclic products
\begin{equation*}
R_{C_1}(A_P),\quad R_{C_2}(A_P),\quad R_{C_3}(A_P),\quad R_{C_4}(A_P)
\end{equation*}
determine $[P]$. On the other hand, multiplying the defining ratios shows that their total product is equal to $1$, and therefore
\begin{equation*}
R_{C_1}(A_P)+R_{C_2}(A_P)+R_{C_3}(A_P)+R_{C_4}(A_P)=0.
\end{equation*}
It follows that $\Ccal(\Gcal)$ is diffeomorphic to an open cell of dimension $3$.

\smallskip
\noindent
\emph{Case (2): exactly one right-angled edge.}
In this case, exactly two links contribute independent parameters, say
\begin{equation*}
R_{C_1}(A_P)\quad\text{and}\quad R_{C_3}(A_P),
\end{equation*}
and the remaining cyclic data are determined by these parameters. Hence $\Ccal(\Gcal)$ is diffeomorphic to an open cell of dimension $2$.

\smallskip
\noindent
\emph{Cases (3) and (4): exactly two right-angled edges.}
There is a unique relevant circuit, which arises either from a vertex link or from the $4$-cycle
\begin{equation*}
C=(1,2,3,4).
\end{equation*}
Hence the deformation space $\Ccal(\Gcal)$ is parametrized by a single normalized cyclic product, so it is diffeomorphic to an open cell of dimension $1$.

\smallskip
\noindent
\emph{Case (5): at least three right-angled edges.}
In this case, the Coxeter diagram has no circuit of length at least $3$. Thus there is no normalized cyclic-product parameter, and $\Ccal(\Gcal)$ consists of a single point.

\smallskip
\noindent
In every case, $\Ccal(\Gcal)$ is diffeomorphic to an open cell of dimension $b(\Gcal)=e_+-3\in\{0,1,2,3\}$, as claimed.
\end{proof}

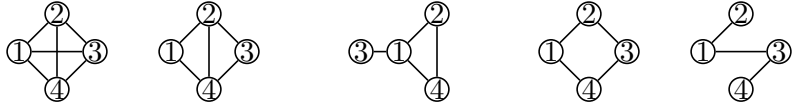
\begin{figure}[ht]
\centering
\begin{tikzpicture}
    \begin{scope}[xshift=0cm]
        \node[vtx] (L)    at (0,0)  {1};
        \node[vtx] (T)    at (0.5,0.5)  {2};
        \node[vtx] (R)    at (1,0)  {3};
        \node[vtx] (B)    at (0.5,-0.5) {4};
        \draw[edg] (L)--(T)--(R)--(B)--(L);
        \draw[edg] (L)--(R);
        \draw[edg] (T)--(B);
    \end{scope}

    \begin{scope}[xshift=2cm]
        \node[vtx] (L)    at (0,0)  {1};
        \node[vtx] (T)    at (0.5,0.5)  {2};
        \node[vtx] (R)    at (1,0)  {3};
        \node[vtx] (B)    at (0.5,-0.5) {4};
        \draw[edg] (L)--(T)--(R)--(B)--(L);
        \draw[edg] (T)--(B);
    \end{scope}
    \begin{scope}[xshift=5cm]
       \node[vtx] (L)    at (0,0)  {1};
        \node[vtx] (T)    at (0.5,0.5)  {2};
        \node[vtx] (R)    at (-0.5,0)  {3};
        \node[vtx] (B)    at (0.5,-0.5) {4};
        \draw[edg] (L)--(T)--(B)--(L);
        \draw[edg] (R)--(L);
    \end{scope}
    \begin{scope}[xshift=7cm]
       \node[vtx] (L)    at (0,0)  {1};
        \node[vtx] (T)    at (0.5,0.5)  {2};
        \node[vtx] (R)    at (1,0)  {3};
        \node[vtx] (B)    at (0.5,-0.5) {4};
        \draw[edg] (L)--(T)--(R)--(B)--(L);
    \end{scope}
    \begin{scope}[xshift=9cm]
       \node[vtx] (L)    at (0,0)  {1};
        \node[vtx] (T)    at (0.5,0.5)  {2};
        \node[vtx] (R)    at (1,0)  {3};
        \node[vtx] (B)    at (0.5,-0.5) {4};
        \draw[edg] (L)--(T);
        \draw[edg] (B)--(R);
        \draw[edg] (L)--(R);
    \end{scope}
    \end{tikzpicture}
\caption{Examples of Cases (1)-(5), from left to right}
\label{fig:tits-examples}
\end{figure}

\begin{remark}\label{rem:3.3}
Cases~(3) and (4) in the previous theorem are called the pan type and the cycle type, respectively. By Appendix~C of \cite{ChoiLeeMarquis22}, every irreducible $2$-perfect labeled simplex of dimension $d\ge 4$ belongs to one of the three types shown in the Figure below.
\end{remark}

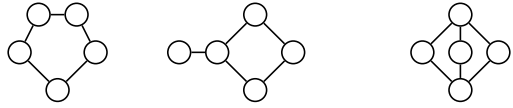
\begin{figure}[ht]
\centering
\begin{tikzpicture}
    \begin{scope}[xshift=0cm]
        \node[vtx] (a1) at (0,0) {};
        \node[vtx] (a2) at (0.25,0.5) {};
        \node[vtx] (a3) at (0.75,0.5) {};
        \node[vtx] (a4) at (1,0) {};
        \node[vtx] (a5) at (0.5,-0.5) {};
        \draw[edg] (a1)--(a2)--(a3)--(a4)--(a5)--(a1);
    \end{scope}

    \begin{scope}[xshift=2.6cm]
        \node[vtx] (b1) at (-0.5,0) {};
        \node[vtx] (b2) at (0,0) {};
        \node[vtx] (b3) at (0.5,0.5) {};
        \node[vtx] (b4) at (1,0) {};
        \node[vtx] (b5) at (0.5,-0.5) {};
        \draw[edg] (b1)--(b2);
        \draw[edg] (b2)--(b3)--(b4)--(b5)--(b2);
    \end{scope}

    \begin{scope}[xshift=5.3cm]
        \node[vtx] (c1) at (0.5,0.5) {};
        \node[vtx] (c2) at (0.5,-0.5) {};
        \node[vtx] (c3) at (0,0) {};
        \node[vtx] (c4) at (0.5,0) {};
        \node[vtx] (c5) at (1,0) {};
        \draw[edg] (c1)--(c3);
        \draw[edg] (c1)--(c4);
        \draw[edg] (c1)--(c5);
        \draw[edg] (c2)--(c3);
        \draw[edg] (c2)--(c4);
        \draw[edg] (c2)--(c5);
    \end{scope}
\end{tikzpicture}
\caption{A $5$-cycle, a $4$-pan and $K_{2,3}$ from left to right.}
\label{fig:graphs}
\end{figure}

Using  Remark~\ref{rem:3.3}, we describe the deformation space of an irreducible,
$2$-perfect labeled simplex of dimension $d\ge4$ by the same method.

\begin{theorem}\label{thm:3.4}
Let $\Gcal$ be an irreducible, $2$-perfect labeled simplex of dimension $d\ge4$,
and let $\Ccal(\Gcal)$ be its deformation space. Then $\Ccal(\Gcal)$ is
homeomorphic to an open cell of dimension
$b(\Gcal)=e_+-d\in\{0,1,2\}$.
\end{theorem}

\begin{proof}
By the above remark, $\Gcal$ is of $K_{2,3}$, pan, or cycle type, so the
conclusion follows by the same argument as in Theorem~\ref{thm:3.2}.
\end{proof}

\subsection{Truncation and once truncated simplices}
In this subsection, we explain the truncation procedure used to construct truncation polytopes. We recall what truncation means in our setting.

Let $\Gcal$ be a combinatorial polytope. A \emph{truncation at a vertex} $v$ is the operation of cutting off $v$ and replacing it with a new facet $s$; we denote the resulting polytope by $\Gcal^{\dagger v}$. A \emph{truncation $d$-polytope} is a polytope obtained from a $d$-simplex by successively truncating vertices. A \emph{once truncated simplex} is a polytope obtained from a simplex by truncating each vertex of a simplex at most once.

In the Coxeter setting, there is a similar criterion for when a vertex is truncatable:

\begin{definition}\label{def:3.5}
Let $P$ be a Coxeter polytope in $\S^d$, let $v$ be a vertex of $P$, and let $S_v$ be the set of facets of $P$ containing $v$. The vertex $v$ is \emph{truncatable} if the projective subspace $\Pi_v=\operatorname{Span}\{b_s : s\in S_v\}$ is a hyperplane, and for each edge $e$ containing $v$, the intersection of $\Pi_v$ with the relative interior of $e$ consists of a single point.
\end{definition}

We denote by $\Pi_v^+$ (resp.\ $\Pi_v^-$) the connected component of $\S^d\setminus \Pi_v$ that does not contain $v$ (resp.\ contains $v$). After truncation, we obtain a new Coxeter polytope $P^{\dagger v}=P\cap \overline{\Pi_v^+}$. By the construction of $\Pi_v$, all poles in $S_v$ are contained in $\Pi_v$ and this implies the following:
\begin{itemize}
    \item the dihedral angles of the ridges in the new facets of $P^{\dagger v}$ are all right angles.
    \item the hyperplane $\Pi_v$ is preserved by the reflections $\{\sigma_s~|~s\in S_v\}$, and $P~\cap ~\Pi_v$ is a Coxeter polytope which is isomorphic to $P_v$.
\end{itemize}

We define truncatability for a set of vertices $\Vcal$ of $P$ as follows: the set $\Vcal$ is \emph{truncatable} if each vertex $v\in\Vcal$ is truncatable and
$P~\cap~\Pi_v~\cap~\Pi_w = \varnothing$
for any distinct pair of vertices $v,w \in \Vcal$.

Recall that a vertex $v$ of $P$ is \emph{simple} if the link $P_v$ is a simplex, and a polytope is \emph{simple} if all its vertices are simple.

\begin{theorem}[{\cite[Prop.~4.14 and Lem.~4.17]{Marquis17_CoxeterHilbert}}]\label{thm:3.6} 
Let $P$ be an irreducible, loxodromic, $2$-perfect Coxeter polytope.
Assume that $\Vcal$ is a set of simple vertices of $P$.
Then $\Vcal$ is truncatable if and only if the vertex link $P_v$ is loxodromic for each $v\in\Vcal$.
\end{theorem}

Recall that a Coxeter group $W_S$ is \emph{Lann\'er} if it is large and $W_{S\setminus\{s\}}$ is spherical for every $s\in S$, and it is \emph{$2$-Lann\'er} if it is irreducible, large, and $W_{S\setminus\{s,t\}}$ is spherical for every distinct $s,t\in S$. See \cite{Lanner50, Koszul67, Chein69} for the classification of Lann\'er groups and \cite{Maxwell82} for the classification of $2$-Lann\'er groups.

Before stating the characterization of truncatability, we recall the consequences from \cite[Lem.~2.15 and Rmk.~2.18]{ChoiLeeMarquis22} that will be used below. A labeled simplex is perfect, irreducible, and large if and only if its Coxeter group is Lann\'er, and it is $2$-perfect, irreducible, and large if and only if its Coxeter group is $2$-Lann\'er. Moreover, a loxodromic perfect Coxeter simplex has Coxeter group either Lann\'er or of affine type $\widetilde A_d$; see \cite[p.~141, Fig.~2]{Humphreys90} for the Lann\'er simplex groups and \cite[p.~34, Fig.~2]{Humphreys90} for affine type $\widetilde A_d$.

\begin{corollary}[{\cite[Cor.~4.7]{ChoiLeeMarquis22}}]\label{cor:3.7}
Let $\Gcal$ be an irreducible, large, $2$-perfect labeled simple polytope of dimension $d\ge 3$,
and let $v$ be a vertex of $\Gcal$. Assume that $[P]\in \Ccal(\Gcal)$. Then:
\begin{itemize}
\item if $v$ is a truncatable vertex of $P$, then $v$ is Lann\'er or $\widetilde A_{d-1}$;
\item if $v$ is Lann\'er, then $v$ is a truncatable vertex of $P$;
\item if $v$ is $\widetilde A_{d-1}$, then $v$ is a truncatable vertex of $P$
if and only if the normalized cyclic product $R_{C_v}(A_P)\neq 0$,
where $C_v$ is a relevant circuit of $W_{G_v}$.
\end{itemize}
\end{corollary}

We say that a vertex $v$ of $\Gcal$ is \emph{of type $\mathsf{T}$} if its link $\Gcal_v$ is \emph{of type $\mathsf{T}$}. For example, saying that a vertex v is of type $\widetilde{A}_{d-1}$ means that its link $\Gcal_v$ is of type $\widetilde{A}_{d-1}$.

The following result is an immediate consequence of the previous corollary.

\begin{corollary}\label{cor:3.8}
    Let $\Gcal$ be an irreducible, large, $2-$perfect labeled simple polytope of dimension $d\ge 3$, $\Vcal_A$ the set of all $\widetilde{A}_{d-1}$ vertices of $\Gcal$, and $\Vcal$ a set of all loxodromic vertices of $\Gcal$.
    Define
    \begin{equation*}
        \Ccal(\Gcal)^{\dagger \Vcal}=\{[P]\in\Ccal(\Gcal)~|~ R_{C_v}(A_P)\neq0 \text{ for each }v\in\Vcal~\cap~\Vcal_A\}.
    \end{equation*}
    Then the map $ \Ccal(\Gcal)^{\dagger \Vcal}\rightarrow \Ccal(\Gcal^{\dagger\Vcal})$ induced by the truncation is a homeomorphism. In particular, if the link $\Gcal_v$ is Lann\`er for each $v\in\Vcal$, then $\Ccal(\Gcal)$ is homeomorphic to $ \Ccal(\Gcal^{\dagger \Vcal})$.
\end{corollary}

\section{Deformation spaces of 2-perfect truncation polytopes}\label{sec:general}
Let $\Gcal$ be an irreducible, $2$-perfect, large truncation polytope, and let $\Ccal(\Gcal)$ denote its deformation space. We identify the cyclic products that provide the relevant parameters for the deformation space.

\subsection{Gluing and splitting along prismatic circuits}
\begin{definition}\label{def:4.1}
Let $\mathcal{G}$ be a combinatorial $d$-polytope and $\mathcal{G}^*$ be the dual polytope. Let $S_{\mathcal{G}}$ be the set of facets of $\mathcal{G}$.
A subset $\delta\subset S_{\mathcal{G}}$ is called a \emph{prismatic circuit} if

\begin{itemize}
    \item the convex hull $\Delta_\delta$ of dual vertices in $\delta^*$ is a $(d-1)$-simplex, and
    \item the relative interior of $\Delta_\delta$ lies in the interior of $\Gcal^*$.
\end{itemize}
\end{definition}

Let $\Gcal_1$ and $\Gcal_2$ be combinatorial $d$-polytopes, and let $v_i$ be a vertex of $\Gcal_i$ for $i \in \{1,2\}$. Let $\Gcal_i^{\dagger v_i}$ denote the polytope obtained by truncating $v_i$. If there exists a poset isomorphism
\begin{equation*}
    \phi : (\Gcal_1)_{v_1} \longrightarrow (\Gcal_2)_{v_2}
\end{equation*}
between the vertex links, then one can glue $G_1^{\dagger v_1}$ and $G_2^{\dagger v_2}$ along the corresponding truncation facets via $\phi$. The set of facets appearing in the link $(\Gcal_i)_{v_i}$ forms a prismatic circuit. We call this procedure the \emph{gluing} of $G_1^{\dagger v_1}$ and $G_2^{\dagger v_2}$.

Conversely, let $\Gcal$ be a combinatorial $d$-polytope and let $\delta$ be a prismatic circuit of $\Gcal$. In the dual polytope $\Gcal^*$, the vertices dual to the facets in $\delta$ span a $(d-1)$-simplex $\Delta_\delta$ whose relative interior lies in $\mathring{\Gcal}^*$. Cutting $G^*$ along $\Delta_\delta$ and dualizing back produces two polytopes $\Gcal_1$ and $\Gcal_2$. The simplex $\Delta_\delta$ becomes a common vertex of $\Gcal_1$ and $\Gcal_2$; we denote the corresponding vertices by $v_1 \in \Gcal_1$ and $v_2 \in \Gcal_2$. We call this procedure the \emph{splitting} of $\Gcal$ along $\delta$.

We now distinguish three types of prismatic circuits for labeled polytopes.
\begin{definition}\label{def:4.2}
    Let $\Gcal$ be a labeled $d-$polytope and $\delta$ be a prismatic circuit of $\Gcal$.
    Suppose $\Gcal$ splits along $\delta$ into  $\mathcal{G}_1^{\dagger v_1}$ and $\mathcal{G}_2^{\dagger v_2}$. A prismatic circuit $\delta$ is
    \begin{itemize}
        \item \emph{useless} if $\mathcal{G}_1^{\dagger v_1}= \mathcal{G}_2^{\dagger v_2}=\Gcal$ and $W_\Gcal=W_\delta \times \tilde{A_1}$;
        \item \emph{non-essential} if there is $i \in \{1,2\}$ such that $\mathcal{G}_i^{\dagger v_i}$ is useless;
        \item \emph{essential}, otherwise. 
    \end{itemize}
\end{definition}

The above definition leads to the following lemma for Coxeter polytopes.

\begin{lemma}[{\cite[Lem~5.4]{ChoiLeeMarquis22}}]\label{lem:4.3}
    Let $\Gcal$ be an irreducible, $2$--perfect, large labeled polytope of dimension $d \ge 3$ and $\delta$ a prismatic circuit of $\Gcal$. Assume $\Gcal$ splits along $\delta$ into $\Gcal_1$ and $\Gcal_2$. If $\Ccal(\Gcal)\neq \varnothing$, then the following hold:
    \begin{enumerate}
        \item the polytopes $\Gcal_1$ and $\Gcal_2$ are $2$--perfect;
        \item the Coxeter group $W_\delta$ is $\tilde{A}_{d-1}$ or Lann\`er;
        \item if $\delta$ is essential, then both $\Gcal_1$ and $\Gcal_2$ are irreducible and large.
    \end{enumerate}
\end{lemma} 

\subsection{Deformation spaces of Coxeter truncation polytopes}

Let $\Gcal$ be an irreducible, large, $2$-perfect labeled polytope of dimension $d \ge 3$, with facet set $S$, and fix an essential prismatic circuit $\delta$ of $\Gcal$. Splitting $\Gcal$ along $\delta$ produces two labeled polytopes $\Gcal_1^{\dagger v_1}$ and $\Gcal_2^{\dagger v_2}$. For $i=1,2$, let $S_i \subset S$ denote the subset of facets of $\Gcal$ corresponding to the facets of $\Gcal_i^{\dagger v_i}$. Then $S = S_1 \cup S_2$ and $S_1 \cap S_2 = \delta.$ We now define the splitting map on deformation spaces.

\begin{definition}\label{def:4.4}
Let $\Gcal$ be an irreducible, $2$--perfect, large labeled $d$--polytope $(d\ge 3)$ with facet set $S$,
and let $\delta$ be an essential prismatic circuit.
Suppose that splitting $\Gcal$ along $\delta$ yields $\Gcal_1^{\dagger v_1}$ and $\Gcal_2^{\dagger v_2}$, and write
$S=S_1\cup S_2$ with $\delta=S_1\cap S_2$.

For $[P]\in \Ccal(\Gcal)$, choose a representative
\begin{equation*}
P=\bigcap_{s\in S}\; \S\bigl(\{x\in V \mid \alpha_s(x)\le 0\}\bigr).
\end{equation*}
For $i=1,2$, set
\begin{equation*}
P_i:=\bigcap_{s\in S_i}\; \S\bigl(\{x\in V \mid \alpha_s(x)\le 0\}\bigr),
\end{equation*}
so that $P=P_1\cap P_2$.
Let $v_i$ be the vertex of $P_i$ determined by the cut along $\delta$, and let $P_i^{\dagger v_i}$ denote the
truncation of $P_i$ at $v_i$.

We define the \emph{splitting map}
\begin{equation*}
\mathrm{Cut}_\delta:\Ccal(\Gcal)\longrightarrow \Ccal(\Gcal_1^{\dagger v_1})\times \Ccal(\Gcal_2^{\dagger v_2})
\end{equation*}
by
\begin{equation*}
\mathrm{Cut}_\delta([P]) := \bigl([P_1^{\dagger v_1}],\,[P_2^{\dagger v_2}]\bigr).
\end{equation*}
\end{definition}

\begin{remark}\label{rem:4.5}
The map $\mathrm{Cut}_\delta$ is well-defined. Indeed, for each $[P]\in\Ccal(\Gcal)$, the subspace
$\Pi_\delta\subset V$ spanned by $(b_s)_{s\in\delta}$ is a hyperplane, and $\S(\Pi_\delta)$ meets the relative interior of each edge in the prismatic poset associated with $\delta$ in exactly one point \cite[Lem.~8.19]{Marquis17_CoxeterHilbert}. Hence the vertices $v_i$ are truncatable, which implies that $P_i^{\dagger v_i}\in \Ccal(\Gcal_i^{\dagger v_i})$. Moreover, the links $(P_1)_{v_1}$ and $(P_2)_{v_2}$ are isomorphic.
\end{remark}

For $i=1,2$, let $\delta_i$ denote the set of facets of $\Gcal_i$ incident to the vertex $v_i$. By Lemma~\ref{lem:4.3}, the Coxeter group $W_\delta$ is either Lann\'er or of affine type $\widetilde{A}_{d-1}$. Therefore $W_{\delta_i}=W_\delta$ is a perfect $(d-1)$-simplex, and hence it is of either tree type or cycle type. 

In the tree-type case, there is no relevant circuit, and we set $ R_{\delta_i}=0.$ In the cycle-type case, there are exactly two relevant circuits, namely an oriented circuit $C_{\delta_i}$ and its opposite oriented circuit $\overline{C}_{\delta_i}$. By choosing the orientation of $C_{\delta_i}$ appropriately, we may arrange that $ R_{\delta_1}\bigl([P_1^{\dagger v_1}]\bigr) = R_{\delta_2}\bigl([P_2^{\dagger v_2}]\bigr).$
This motivates the following compatibility condition:
\begin{equation*}
    \Ccal(\Gcal_1^{\dagger v_1}) \boxtimes_\delta \Ccal(\Gcal_2^{\dagger v_2}):= \left\{\bigl([P_1^{\dagger v_1}],[P_2^{\dagger v_2}]\bigr)\middle|R_{\delta_1}\bigl([P_1^{\dagger v_1}]\bigr)=R_{\delta_2}\bigl([P_2^{\dagger v_2}]\bigr)\right\}.
\end{equation*}

The next lemma identifies the fibers of the splitting map with bending orbits. More precisely, once the two truncated pieces are fixed, the remaining freedom in reconstructing the original polytope is given by a one--parameter bending deformation along the cut.

\begin{lemma}[{\cite[Lem~4.36]{Marquis10_EspacePolyedres}} and {\cite[Lem~5.5]{ChoiLeeMarquis22}}]\label{lem:4.6}
Let $\Gcal$ be an irreducible, large, $2$--perfect labeled polytope of dimension $d\ge3$ and $\delta$ an essential prismatic circuit of $\Gcal$. Suppose $\Gcal$ splits along $\delta$ into $\Gcal_1^{\dagger v_1}$ and $\Gcal_2^{\dagger v_2}$. Then there exists an $\R$--action $\Psi$ on $\Ccal(\Gcal)$ such that $\mathrm{Cut}_\delta$ is a $\Psi$--invariant fibration onto $\Ccal(\Gcal_1^{\dagger v_1})\boxtimes_{\delta}\Ccal(\Gcal_2^{\dagger v_2})$ and $\Psi$ is simply transitive on each fiber of $\mathrm{Cut}_\delta$.
\end{lemma}

\begin{remark}\label{rem:4.7}
We include a proof for completeness. Our argument follows the proofs of \cite[Lem~5.5]{ChoiLeeMarquis22} and \cite[Lem~4.36]{Marquis10_EspacePolyedres}, with a minor modification in the choice of a relevant tuple: when necessary, we use $(s_\ell,s_1,\dots,s_j,s_r)$ rather than only $(s_\ell,s_1,s_r)$.
\end{remark}

\begin{proof}
Choose coordinates adapted to the splitting along $\delta$. Given $[P] \in \Ccal(\Gcal)$, cut $P$ along $\delta$ to obtain two polytopes $P_1^{\dagger v_1}$ and $P_2^{\dagger v_2}$ such that $P=P_1\cap P_2$ and $\mathrm{Cut}_\delta([P])=\bigl([P_1^{\dagger v_1}], [P_2^{\dagger v_2}]\bigr).$

Let $\{e_1,\dots,e_{d+1}\}$ be the standard basis of $\R^{d+1}$, and let $\{e_1^*,\dots,e_{d+1}^*\}$ be the corresponding dual basis. Up to a change of basis, we may assume that the supporting hyperplanes of the facets in $\delta$ are given by $\{\S(\ker(e_i^*))\}_{i=1}^d$, and that the subspace $\Pi_\delta^P$ spanned by the vectors $(b_s)_{s\in\delta}$ is equal to $\ker(e_{d+1}^*)$.

For each $u\in\R$, define $g_u:=\operatorname{diag}(e^u,\dots,e^u,e^{-du}) \in \SL^\pm(d+1,\R)$. We define an $\R$--action $\Psi$ on $\Ccal(\Gcal)$ by
\begin{equation*}
\Psi_u([P]) := [P_1\cap g_u(P_2)].
\end{equation*}
Since the class $[P_2] \in \Ccal(\Gcal_2)$ is invariant under conjugation by $g_u$, the point $\Psi_u([P])$ lies in the same fiber of $\mathrm{Cut}_\delta$ as $[P]=\Psi_0([P])$. Hence $\Psi_u$ preserves each fiber of $\mathrm{Cut}_\delta$.

Let $s_1, s_2, \dots, s_d$ be the facets contained in the circuit $\delta$. We can choose a facet $s_\ell$ of $\Gcal_1$ and a facet $s_r$ of $\Gcal_2$, both not in $\delta$. By rearranging indices, there exists a subset $\{1, \dots, j\}$ in $\{1,\dots,d\}$ such that $(s_\ell,s_{1},\dots,s_{j},s_r)$ is relevant.

After another change of basis, we may assume that the facet $s_\ell$ is defined by the functional $-e_1^*-e_2^*-\cdots-e_{d+1}^*$, and that the facet $s_r$ is defined by $-\lambda_1e_1^*-\lambda_2e_2^*-\cdots -\lambda_de_d^*+\lambda_{d+1}e_{d+1}^*$,
where $\lambda_i>0$ for all $i=1,\dots,d+1$.

For $i=1,\dots,d$, let $\sigma_i:=\sigma_{s_i}=\mathrm{Id}-\alpha_{s_i}\otimes b_{s_i}, \alpha_i:=\alpha_{s_i}=-e_i^*.$ Let $\alpha_{d+1}$ denote the functional defining the cutting hyperplane $\Pi_\delta$. We choose its sign so that $\alpha_{d+1}=-e_{d+1}^*$. Let $u_{s_\ell}$ and $w_{s_\ell}$ be the projections of $b_{s_\ell}$ onto $\ker(e_{d+1}^*)$ and $\operatorname{Span}(e_{d+1})$, respectively. Similarly, let $u_{s_r}$ and $w_{s_r}$ be the corresponding projections of $b_{s_r}$.

We use superscripts to denote the data associated with $\Psi_u([P])$. Then the functionals are given by
\begin{equation*}
    \begin{aligned}
\alpha_{s_\ell}^u &= \alpha_1 + \alpha_2 + \cdots+ \alpha_{d+1}, \\
\alpha_{i}^u &= \alpha_{i} \quad \text{for } i = 1, 2, \ldots,d \\
\alpha_{s_r}^u &= e^{-u}(\lambda_1 \alpha_1 + \lambda_2 \alpha_2 + \cdots + \lambda_d \alpha_d) - e^{du} \lambda_{d+1} \alpha_{d+1}.\\
\end{aligned}
\end{equation*}

Similarly, the polars are also given by 
\begin{equation*}
    \begin{aligned}
b_{s_\ell}^u &= u_{s_\ell}+w_{s_\ell}, \\
b_{i}^u &= b_{i} \quad \text{ for } i = 1, 2, \ldots,d \\
b_{s_r}^u &= e^{u}u_{s_r} + e^{-du} w_{s_r}.\\
\end{aligned}
\end{equation*}

Now let $C=(s_\ell,s_{i_1},\dots,s_{i_j},s_r)$. A direct computation shows that the two cyclic products corresponding to $C$ and $\overline C$ are
\begin{equation*}
    \begin{aligned}
&\alpha_{s_\ell}^u(b_1^u)\alpha_1^u(b_2^u)\dots\alpha_j^u(b_{s_r}^u)\alpha_{s_r}^u(b_{s_\ell}^u)=K_1(x_1+e^{(d+1)u}y_1)\text{ and }\\ &\alpha_{s_\ell}^u(b_{s_r}^u)\alpha_{s_r}^u(b_j^u)\cdots\alpha_2(b_{1}^u)\alpha_{1}^u(b_{s_\ell}^u)=K_2(x_2+e^{-(d+1)u}y_2)
\end{aligned}
\end{equation*}

where

\begin{equation*}
\begin{aligned}
    K_1 & = -\alpha_{s_\ell}(b_1) \cdots \alpha_{j}(b_{s_r})\\
    x_1 & = -(\lambda_1 \alpha_1 + \lambda_2 \alpha_2 +\cdots + \lambda_d \alpha_d)(u_{s_\ell})\\
    y_1 & = \lambda_{d+1} \alpha_{d+1}(w_{s_\ell})\\
    K_2 & = -\alpha_{s_r}(b_j) \cdots \alpha_{1}(b_{s_\ell})\\
    x_2 & =- (\alpha_1 + \alpha_2 + \cdots+ \alpha_{d+1})(u_{s_r})\\
    y_2 & = - \alpha _{d+1}(w_{s_r}).
\end{aligned}
\end{equation*}

Observe that both $K_1$ and $K_2$ have the same sign. By the choice of an adapted basis, the quantities $x_1,x_2,y_1$, and $y_2$ are all positive.

Let $A^u$ be the Cartan matrix of $\Psi_u([P])$. Define a map $R:\R\to\R$ by
\begin{equation*}
u \longmapsto R_C(A^u)=\log\!\left(\frac{K_1(x_1+e^{(d+1)u}y_1)}{K_2(x_2+e^{-(d+1)u}y_2)}\right).
\end{equation*}
Since $x_1,x_2,y_1,y_2>0$ and $K_1/K_2 >0$, the map $R$ is a homeomorphism. In particular, the action of $\Psi$ on each fiber of $\mathrm{Cut}_\delta$ is free.

Now suppose that $[P],[Q]\in\Ccal(\Gcal)$ satisfy
$\mathrm{Cut}_\delta([P])=\mathrm{Cut}_\delta([Q])$,
with $P=P_1\cap P_2$ and $Q=Q_1\cap Q_2$. Then we may assume that $Q_1=P_1$ and $Q_2=g(P_2)$ for some $g\in\SL^\pm(d+1,\R)$. Since the $d$ facets in $\delta$ are common to both pieces, the element $g$ commutes with the corresponding reflections $\sigma_1,\dots,\sigma_d$. Moreover, $g$ acts trivially on $\S(\Pi_\delta)$. Equivalently, it fixes $\Pi_\delta$ pointwise in projective space. It follows that $g$ lies in the one--parameter centralizer of $\langle \sigma_1,\dots,\sigma_d\rangle$, and hence
\begin{equation*}
g=g_\lambda:=\operatorname{diag}(\lambda,\dots,\lambda,\lambda^{-d})
\qquad \text{for some } \lambda>0.
\end{equation*}
Therefore $[Q]=\Psi_{\log\lambda}([P])$, so $\Psi$ is transitive on each fiber.

Finally, given $\bigl([P_1^{\dagger v_1}],[P_2^{\dagger v_2}]\bigr) \in \Ccal(\Gcal_1^{\dagger v_1})\boxtimes_{\delta}\Ccal(\Gcal_2^{\dagger v_2})$, we define $P=P_1\cap P_2$. Let $\delta_i$ be the set of facets of $\Gcal_i$ containing the vertex $v_i$. Then we obtain $R_\delta\bigl([P]\bigr)=R_{\delta_1}\bigl([P_1^{\dagger v_1}]\bigr)=R_{\delta_2}\bigl([P_2^{\dagger v_2}]\bigr)$ and so $\mathrm{Cut}_\delta$ is surjective onto $\Ccal(\Gcal_1^{\dagger v_1})\boxtimes_{\delta}\Ccal(\Gcal_2^{\dagger v_2})$.

\end{proof}

\begin{theorem}[{\cite[Thm~A]{ChoiLeeMarquis22}}\label{thm:4.8}, {\cite{Marquis10_EspacePolyedres}}]
    Let $\Gcal$ be an irreducible, large, $2$--perfect labeled truncation polytope of dimension $d\ge3$ and let $\Ccal(\Gcal)$ be the deformation space of $\Gcal$. Assume that $\Ccal(\Gcal)$ is nonempty. Then the following hold:
    \begin{itemize}
        \item the dimension $d\le9$.
        \item $\Ccal(\Gcal)$ is a union of finitely many open cells of dimension $e^+-d$ where $e^+$ is the number of non--right--angled ridges.
        \item $\Gcal$ admits a hyperbolic structure if and only if $\Ccal(\Gcal)$ is connected.
    \end{itemize}
    \end{theorem}

The proof of Theorem~\ref{thm:4.8} is based on the fact that an irreducible, large, $2$-perfect labeled truncation polytope can be decomposed along prismatic circuits into once-truncated simplices.
    
\begin{remark}\label{rem:4.9}
For an irreducible, large, $2$-perfect labeled truncation polytope with nonempty deformation space, $\Ccal(\Gcal)$ is connected if and only if $\Gcal$ has no essential
prismatic circuit whose Coxeter subgroup is of type $\widetilde A_{d-1}$. Equivalently, the existence of such a circuit is precisely the obstruction to connectedness.
\end{remark}

\section{Proof of main theorem}\label{sec:proof}

\subsection{Basic cases}

We first prove the finiteness result for simplices and once-truncated simplices, which serve as the basic cases of our argument. The key point is that integrality forces the relevant cyclic products to lie in a finite set, and hence only finitely many integral representations can occur in the deformation space.

\begin{theorem}\label{thm:5.1}
Let $\Gcal$ be an irreducible, large, $2$-perfect labeled simplex of dimension $d \ge 3$, and let $\Ccal(\Gcal)$ be its deformation space. Then $\Ccal(\Gcal)$ contains only finitely many integral Vinberg representations.
\end{theorem}

\begin{proof}
Let $[P] \in \Ccal(\Gcal)$ be an integral Vinberg representation. By Lemma~\ref{lem:2.17}, every cyclic product of the Cartan matrix $A_P$ is an integer. If $\Gcal$ has an edge labeled by $m \notin \{2,3,4,6\}$, then
\begin{equation*}
a_{ij}a_{ji}=4\cos^2(\pi/m)\notin \Z
\end{equation*}
for some distinct $i,j$, which contradicts the integrality assumption. Hence either $\Ccal(\Gcal)$ contains no integral Vinberg representations, or every edge label of $\Gcal$ belongs to $\{2,3,4,6\}$.

Now let $C=(i_1,\dots,i_k)$ be a relevant circuit of the Coxeter diagram of $\Gcal$, and let $\overline{C}=(i_k,\dots,i_1)$ be the opposite circuit. Since $\Gcal$ is $2$-perfect, no edge of $C$ is labeled by $2$ or $\infty$. Thus, for each consecutive pair $(i_j,i_{j+1})$, we have
\begin{equation*}
a_{i_ji_{j+1}}a_{i_{j+1}i_j}=4\cos^2(\pi/m_j)\in\{1,2,3\}.
\end{equation*}
Rearranging the factors, we obtain
\begin{equation*}
C(A_P)\,\overline{C}(A_P)
=
\prod_{j=1}^k \bigl(a_{i_ji_{j+1}}a_{i_{j+1}i_j}\bigr)
=:M_C,
\end{equation*}
where $M_C$ is a positive integer determined only by the labeling of $\Gcal$.

By Lemma~\ref{lem:2.17} again, both $C(A_P)$ and $\overline{C}(A_P)$ lie in $\Z$. Hence $C(A_P)$ divides $M_C$, and therefore
\begin{equation*}
\overline{C}(A_P)=\frac{M_C}{C(A_P)}.
\end{equation*}
It follows that there are only finitely many possibilities for the pair
\begin{equation*}
\bigl(C(A_P),\overline{C}(A_P)\bigr)\in\Z^2,
\end{equation*}
and hence only finitely many possibilities for the corresponding normalized cyclic product.

Finally, Theorems~\ref{thm:3.2} and~\ref{thm:3.4} show that the deformation space $\Ccal(\Gcal)$ is determined by finitely many normalized cyclic products associated with relevant circuits. Under the integrality assumption, each of these parameters can take only finitely many values. Therefore $\Ccal(\Gcal)$ contains only finitely many integral Vinberg representations.
\end{proof}

\begin{remark}\label{rem:5.2}
Since the integrality condition forces every cyclic product to lie in $\Z$, the number of integral Vinberg representations in $\Ccal(\Gcal)$ is bounded above by the number of possible values of the corresponding normalized cyclic products.
\end{remark}

\begin{remark}\label{rem:5.3}
If $\Gcal$ is a large labeled triangle that is $2$-perfect but not perfect, then $\Ccal(\Gcal)$ contains infinitely many integral Vinberg representations. Indeed, in contrast to the perfect case, $\Gcal$ may admit an edge labeled by $\infty$, which gives rise to an infinite family under the integrality condition.
\end{remark}

\begin{theorem}\label{thm:5.4}
Let $\Scal$ be an irreducible, large, $2$-perfect labeled simplex of dimension $d \ge 3$, and let $\Vcal$ be the set of all Lann\'er or $\widetilde{A}_{d-1}$ vertices of $\Scal$. Then $\Ccal(\Scal)^{\dagger \Vcal}$ admits only finitely many integral Vinberg representations.
\end{theorem}

\begin{proof}
By Corollary~\ref{cor:3.8}, the deformation space $\Ccal(\Scal)^{\dagger \Vcal}$ is obtained from $\Ccal(\Scal)$ by imposing the additional condition $R_{C_v}(A_P)\neq 0$ for every vertex $v\in V$ of type $\widetilde{A}_{d-1}$. In particular, $\Ccal(\Scal)^{\dagger \Vcal}$ is an open subset of $\Ccal(\Scal)$ obtained by removing finitely many hypersurfaces defined by equations of the form $R_{C_v}(A_P)=0$.

By Theorem~\ref{thm:5.1}, the set of integral Vinberg representations in $\Ccal(\Scal)$ is finite. Intersecting this finite set with the open subset $\Ccal(\Scal)^{\dagger \Vcal}$ still yields a finite set. Therefore $\Ccal(\Scal)^{\dagger \Vcal}$ admits only finitely many integral Vinberg representations.
\end{proof}

\subsection{Nonexistence results}

Let $\Gcal$ be an irreducible, large, $2$-perfect labeled polytope. Before proving the
nonexistence results, we fix a convention that will be used throughout the rest of this
section. Definition~\ref{def:4.4} defines the splitting map $\mathrm{Cut}_\delta$ with values in
$\Ccal(\Gcal_1^{\dagger v_1}) \boxtimes_\delta \Ccal(\Gcal_2^{\dagger v_2})$. Via Corollary~\ref{cor:3.8}, we shall regard its image as lying in $\Ccal(\Gcal_1)^{\dagger v_1} \boxtimes_\delta \Ccal(\Gcal_2)^{\dagger v_2}$ throughout this section. This is the natural viewpoint for integrality, since $W_{\Gcal_i}$ is a standard subgroup
of $W_{\Gcal}$, whereas $W_{\Gcal_i^{\dagger v_i}}$ need not be. 

Let $\Vcal_{\widetilde{A}}$ be the set of all $\widetilde{A}_{d-1}$-vertices of $\Gcal$, and let $\Vcal$ be a set of loxodromic vertices of $\Gcal$. If $\Vcal_{\widetilde{A}} \cap \Vcal \neq \varnothing$, then one obtains the following nonexistence result.

\begin{theorem}\label{thm:5.5}
Let $\Gcal$ be an irreducible, large, $2$-perfect labeled polytope of dimension $d \ge 3$. Let $\Vcal_{\widetilde{A}}$ be the set of all $\widetilde{A}_{d-1}$ vertices of $\Gcal$, and let $\Vcal$ be a set of loxodromic vertices of $\Gcal$. If $\Vcal_{\widetilde{A}}\cap \Vcal\neq \varnothing$, then $\Ccal(\Gcal)^{\dagger \Vcal}$ admits no integral Vinberg representation.
\end{theorem}

\begin{proof}
Suppose for contradiction that there exists $[P]\in \Ccal(\Gcal)^{\dagger \Vcal}$ admitting an integral Vinberg representation. Choose a vertex $v\in \Vcal_{\widetilde{A}}\cap \Vcal$. Then the link $\Gcal_v$ is of type $\widetilde{A}_{d-1}$, and hence its Coxeter group $W_{\Gcal_v}$ is affine of type $\widetilde{A}_{d-1}$. Since $\widetilde{A}_{d-1}$ is of cycle type, we may take the two oriented relevant circuits
\begin{equation*}
C=(1,2,\dots,d)
\qquad\text{and}\qquad
\overline{C}=(d,d-1,\dots,1).
\end{equation*}

Since every edge of $C$ has label $3$, we have
\begin{equation*}
a_{i,i+1}a_{i+1,i}=1
\qquad (i=1,\dots,d-1),
\qquad
a_{d,1}a_{1,d}=1.
\end{equation*}
Therefore
\begin{equation*}
C(A_P)\,\overline{C}(A_P)=1.
\end{equation*}
By integrality, both $C(A_P)$ and $\overline{C}(A_P)$ are integers. Hence
\begin{equation*}
\bigl(C(A_P),\overline{C}(A_P)\bigr)\in\{(1,1),(-1,-1)\}.
\end{equation*}
It follows that
\begin{equation*}
R_C(A_P)=\log\!\left(\frac{C(A_P)}{\overline{C}(A_P)}\right)=0.
\end{equation*}
This contradicts Corollary~\ref{cor:3.8}, which asserts that $R_C(A_P)\neq 0$ for every $[P]\in \Ccal(\Gcal)^{\dagger \Vcal}$. Therefore $\Ccal(\Gcal)^{\dagger \Vcal}$ admits no integral Vinberg representation.
\end{proof}

When $\Gcal$ has an essential prismatic circuit $\delta$ whose Coxeter subgroup is of type $\widetilde{A}_{d-1}$, we obtain a similar nonexistence statement.

\begin{theorem}\label{thm:5.6}
Let $\Gcal$ be an irreducible, large, $2$-perfect labeled polytope of dimension $d \ge 3$, and let $\delta$ be an essential prismatic circuit. If the Coxeter group $W_\delta$ is of type $\widetilde{A}_{d-1}$, then $\Ccal(\Gcal)$ admits no integral Vinberg representation.
\end{theorem}

\begin{proof}
Suppose for contradiction that there exists an integral Vinberg representation $[P]\in \Ccal(\Gcal)$, and write $\mathrm{Cut}_\delta([P])=([P_1],[P_2]),$ with $[P_i]\in \Ccal(\Gcal_i)^{\dagger v_i}$.

Since $W_{\Gcal_1}$ is generated by a subset of the standard generators of $W_\Gcal$, the restriction of the integral Vinberg representation associated to $[P]$ to $W_{\Gcal_1}$ is again integral. By Lemma~\ref{lem:4.6}, the normalized cyclic product $R_{\delta}([P])$ is identified with
\begin{equation*}
R_{\delta_1}([P_1])=R_{\delta_2}([P_2]),
\end{equation*}
where $\delta_i$ denotes the set of facets of $\Gcal_i$ incident to $v_i$.

Since $W_\delta$ is of type $\widetilde{A}_{d-1}$, so is $W_{\delta_1}$. Thus $v_1$ is of type $\widetilde{A}_{d-1}$, and exactly as in the proof of Theorem~\ref{thm:5.5}, integrality implies $R_{\delta_1}([P_1])=0$. On the other hand, Corollary~\ref{cor:3.8} gives $R_{\delta_1}([P_1])\neq 0$, a contradiction.
\end{proof}

Let $\Gcal$ be a labeled polytope. We say that $\Gcal$ is \emph{hyperbolizable} if there exists a hyperbolic Coxeter polytope realizing $\Gcal$. By Remark~\ref{rem:4.9}, the existence of a Coxeter subgroup of type $\widetilde{A}_{d-1}$ is related to the hyperbolizability of $\Gcal$. We therefore obtain the following corollary.

\begin{corollary}\label{cor:5.7}
Let $\Gcal$ be an irreducible, large, $2$-perfect labeled truncation polytope of dimension $d \ge 3$, and let $\Ccal(\Gcal)$ be its deformation space. If $\Gcal$ is not hyperbolizable, then $\Ccal(\Gcal)$ admits no integral Vinberg representation.
\end{corollary}

\begin{proof}
This follows immediately from Remark~\ref{rem:4.9} and Theorem~\ref{thm:5.6}.
\end{proof}

\subsection{The truncation case}

To prove the main theorem, we begin by controlling the bending parameter associated with each essential prismatic circuit. More precisely, when restricted to the integral Vinberg representations in $\Ccal(\Gcal)$, the bending parameter remains bounded on each fiber of $\mathrm{Cut}_\delta$.

\begin{lemma}\label{lem:5.8}
Let $\Gcal$ be an irreducible, large, $2$-perfect labeled polytope of dimension $d \ge 3$, and let $\delta$ be an essential prismatic circuit of $\Gcal$. Suppose that $\Gcal$ splits along $\delta$ into $\Gcal_1^{\dagger v_1}$ and $\Gcal_2^{\dagger v_2}$. Let $\Ccal(\Gcal)(\Z)\subset \Ccal(\Gcal)$ denote the subset of integral Vinberg representations. Then, for each fiber of $\mathrm{Cut}_\delta$, the set of bending parameters corresponding to points of $\Ccal(\Gcal)(\Z)$ is finite.
\end{lemma}

\begin{proof}
Fix a fiber
\begin{equation*}
F=\mathrm{Cut}_\delta^{-1}([P_1],[P_2])
\end{equation*}
and a point $[P]\in F$. Using the bending parametrization of $F$, we write each point of $F$ as $\Psi_u([P])$ for some $u\in\R$. For such $u$, set
\begin{equation*}
N(u):=K_1(x_1+e^{(d+1)u}y_1),
\qquad
D(u):=K_2(x_2+e^{-(d+1)u}y_2).
\end{equation*}
If $\Psi_u([P])\in \Ccal(\Gcal)(\Z)$, then $N(u)$ and $D(u)$ are cyclic products of $\Psi_u([P])$. Hence, by Lemma~\ref{lem:2.17}, $N(u),D(u)\in\Z$.

As $u\to -\infty$, we have $N(u)\to K_1x_1$. Since $y_1>0$, the quantity $N(u)-K_1x_1$ always has the same sign as $K_1$. Therefore there exist $\varepsilon_->0$ and $u_-\in\R$ such that the open interval $I_-$ with endpoints $K_1x_1$ and $K_1x_1+\operatorname{sgn}(K_1)\varepsilon_-$
is disjoint from $\Z$, and $N(u)\in I_-$
for all $u\le u_-$. It follows that
\begin{equation*}
\Psi_u([P])\notin \Ccal(\Gcal)(\Z)
\qquad\text{for all }u\le u_-.
\end{equation*}

Similarly, as $u\to +\infty$, we have $D(u)\to K_2x_2$. Since $y_2>0$, the quantity $D(u)-K_2x_2$ always has the same sign as $K_2$. Hence there exist $\varepsilon_+>0$ and $u_+\in\R$ such that the open interval $I_+$ with endpoints $K_2x_2$ and
$K_2x_2+\operatorname{sgn}(K_2)\varepsilon_+$
is disjoint from $\Z$, and $D(u)\in I_+$
for all $u\ge u_+$. Therefore
\begin{equation*}
\Psi_u([P])\notin \Ccal(\Gcal)(\Z)
\qquad\text{for all }u\ge u_+.
\end{equation*}

Thus $\{u\in\R:\Psi_u([P])\in \Ccal(\Gcal)(\Z)\}\subset [u_-,u_+]$, so the bending parameter is bounded on the fiber. Moreover, since $y_1,y_2>0$, the functions $N(u)$ and $D(u)$ are continuous and strictly monotone in $u$. Hence the set $\{u\in\R:\Psi_u([P])\in \Ccal(\Gcal)(\Z)\}$ is discrete. Since it is contained in the compact interval $[u_-,u_+]$, it is finite.
\end{proof}

Before proving the main theorem, we establish a slightly more general finiteness statement, which also includes noncompact and infinite-volume cases.

\begin{theorem}\label{thm:5.9}
Let $\Gcal$ be an irreducible, large, $2$-perfect truncation polytope of dimension $d \ge 3$. Let $\Ccal(\Gcal)(\Z)\subset \Ccal(\Gcal)$ denote the subset of integral Vinberg representations. Then $\Ccal(\Gcal)(\Z)$ is a finite set.
\end{theorem}

\begin{proof}
We prove the theorem by induction on the number $n$ of essential prismatic circuits of $\Gcal$.

If $n=0$, then $\Gcal$ is either a simplex or a once-truncated simplex, since a truncation polytope with no essential prismatic circuit has no further nontrivial splitting. Therefore the conclusion follows from Theorems~\ref{thm:5.1} and~\ref{thm:5.4}.

Assume now that the statement holds for every irreducible, large, $2$-perfect truncation polytope with at most $k$ essential prismatic circuits, and let $\Gcal$ have $k+1$ essential prismatic circuits. Choose an essential prismatic circuit $\delta$, and suppose that splitting along $\delta$ decomposes $\Gcal$ into $\Gcal_1^{\dagger v_1}$ and $\Gcal_2^{\dagger v_2}$.

Assume that $\Ccal(\Gcal)(\Z)\neq\varnothing$. Then Theorem~\ref{thm:5.6} shows that $W_\delta$ is not of type $\widetilde A_{d-1}$. Since Lemma~\ref{lem:4.3} implies that $W_\delta$ is either Lann\'er or of type $\widetilde A_{d-1}$, it follows that $W_\delta$ is Lann\'er. Hence the cut vertices $v_1$ and $v_2$ are of Lann\'er type, and Corollary~\ref{cor:3.8} identifies $\Ccal(\Gcal_1^{\dagger v_1}) \boxtimes_\delta \Ccal(\Gcal_2^{\dagger v_2})$ with $\Ccal(\Gcal_1)\boxtimes_\delta \Ccal(\Gcal_2)$.

We regard the splitting map as $\mathrm{Cut}_\delta:\Ccal(\Gcal)\longrightarrow\Ccal(\Gcal_1)\boxtimes_\delta \Ccal(\Gcal_2)$. Since $\Gcal$ is a truncation polytope, each of the pieces $\Gcal_1$ and $\Gcal_2$ is again a truncation polytope. Moreover, by Lemma~\ref{lem:4.3}, both $\Gcal_1$ and $\Gcal_2$ are irreducible, large, and $2$-perfect. Each of $\Gcal_1$ and $\Gcal_2$ has at most $k$ essential prismatic circuits. Hence, by the induction hypothesis, the sets $\Ccal(\Gcal_1)(\Z)$ and $\Ccal(\Gcal_2)(\Z)$ are finite. Therefore $\Ccal(\Gcal_1)(\Z)\boxtimes_\delta \Ccal(\Gcal_2)(\Z)$ is finite.

Now let $[P]\in \Ccal(\Gcal)(\Z)$. Since $W_{\Gcal_1}$ and $W_{\Gcal_2}$ are standard subgroups of $W_{\Gcal}$, the restrictions of the Vinberg representation of $[P]$ to these subgroups are again integral. Hence 
\begin{equation*}
    \mathrm{Cut}_\delta([P])\in \Ccal(\Gcal_1)(\Z)\boxtimes_\delta \Ccal(\Gcal_2)(\Z).
\end{equation*}
It follows that only finitely many fibers of $\mathrm{Cut}_\delta$ can contain integral Vinberg representations.

Fix one such fiber $F$. By Lemma~\ref{lem:4.6}, the fiber $F$ is an $\R$-orbit for the bending action, so we may write $F=\{\Psi_u([P_0]) : u\in\R\}$ for some $[P_0]\in F$. By Lemma~\ref{lem:5.8}, the set $\{u\in\R : \Psi_u([P_0])\in \Ccal(\Gcal)(\Z)\}$ is finite. Therefore $F\cap \Ccal(\Gcal)(\Z)$ is finite. Since only finitely many fibers contain integral Vinberg representations and each such fiber contains only finitely many of them, the set $\Ccal(\Gcal)(\Z)$ is finite.
\end{proof}

Theorem~\ref{thm:5.9} establishes the desired finiteness statement for irreducible, large, $2$-perfect truncation polytopes. We may now deduce the main theorem.

\begin{proof}[Proof of Theorem~\ref{thm:1.1}]
Let $P$ be a compact hyperbolic truncation polytope, and let $M$ be the associated Coxeter orbifold with labeled fundamental domain $\Gcal$. Set $\Gamma:=\pi_1^{\mathrm{orb}}(M)$. By Remark~\ref{rem:2.11}, the geometric component $\mathscr{G}(\Gamma,G)$ is identified with the component $\Ccal_0(\Gcal)\subset \Ccal(\Gcal)$ containing the hyperbolic point. Since $\Ccal(\Gcal)$ is connected for compact hyperbolic truncation polytopes by Theorem~\ref{thm:4.8}, it follows that $\Ccal_0(\Gcal)=\Ccal(\Gcal)$. Under this identification, integral points of $\mathscr{G}(\Gamma,G)$ correspond precisely to integral Vinberg representations in $\Ccal(\Gcal)$.

By Theorem~\ref{thm:5.9}, the set $\Ccal(\Gcal)(\Z)$ of integral Vinberg representations is finite. Consequently, the geometric component $\mathscr{G}(\Gamma,G)$ contains only finitely many integral points.
\end{proof}

\begin{remark}\label{rem:5.10}
By Lemma~\ref{lem:5.8} and Theorem~\ref{thm:5.9}, whenever $\Gcal$ is a non-truncation polytope admitting a truncatable vertex and $\Ccal(\Gcal)$ contains only finitely many integral Vinberg representations, one can produce infinitely many non-truncation polytopes by truncating such a vertex and gluing the resulting polytope to a truncation polytope. The figure below illustrates a polytope satisfying these assumptions.
\end{remark}

\begin{figure}[H]
  \centering
 \begin{tikzpicture}[line cap=round,line join=round,scale=1]

    \coordinate (A) at (0,0);
    \coordinate (B) at (4,0);
    \coordinate (C) at (4,4);
    \coordinate (D) at (0,4);

    \coordinate (a) at (1.5,1.5);
    \coordinate (b) at (2.5,1.5);
    \coordinate (c) at (2.5,2.5);
    \coordinate (d) at (1.5,2.5);

    \draw[line width=0.8pt] (A)--(B)--(C)--(D)--cycle;
    \draw[line width=0.8pt] (a)--(b)--(c)--(d)--cycle;

    \draw[line width=0.8pt] (D)--(d);
    \draw[line width=0.8pt] (C)--(c);
    \draw[line width=0.8pt] (A)--(a);
    \draw[line width=0.8pt] (B)--(b);

    \node at (2,2) {$F_1$};
    \node at (2,3.25) {$F_2$};
    \node at (0.75,2.25) {$F_3$};
    \node at (2,0.75) {$F_4$};
    \node at (3.25,2.25) {$F_5$};
    \node at (4.75,2.25) {$F_6$};

    \node at (-0.75,2.25) {\phantom{$F_6$}};

    \node[above] at (2,4) {$4$};
    \node[right] at (4,2) {$4$};
    \node[below] at (2,0) {$2$};
    \node[left]  at (0,2) {$2$};

    \node at (0.9,3.2) {$2$};
    \node at (3.1,3.2) {$4$};
    \node at (0.9,0.8) {$4$};
    \node at (3.1,1.0) {$2$};

    \node[above] at (2,2.5) {$2$};
    \node[left]  at (1.5,2) {$2$};
    \node[below] at (2,1.5) {$2$};
    \node[right] at (2.5,2) {$2$};
    \node[above right, xshift=1pt, yshift=1pt] at (C) {$v$};
  \end{tikzpicture}
  \caption{A labeled cube whose deformation space contains only finitely many integral Vinberg representations.}
  \label{fig:cube-schematic}
\end{figure}
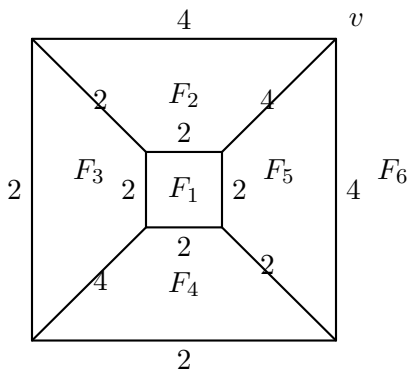

\printbibliography
\end{document}